\documentclass{siamltex}

\usepackage{amsfonts, amsmath, amssymb, psfrag, graphicx, bbm, mathrsfs, xfrac}

\usepackage[colorlinks=true]{hyperref}
\hypersetup{urlcolor=blue, citecolor=blue, linkcolor=blue}

\newtheorem{example}[theorem]{Example}
\newtheorem{remark}[theorem]{Remark}

\def\fin{\ifmmode{\Large$\diamond$}\else{\unskip\nobreak\hfil
\penalty50\hskip1em\null\nobreak\hfil{\Large$\diamond$}
\parfillskip=0pt\finalhyphendemerits=0\endgraf}\fi}

\def\be#1#2\ee{\begin{equation}\label{eq:#1}#2\end{equation}}
\def\req#1{{\rm(\ref{eq:#1})}}
\def\bdm  {\begin{displaymath}}
\def\edm  {\end{displaymath}}
\def\bdmal{\begin{displaymath}\begin{aligned}}
\def\edmal{\end{aligned}\end{displaymath}}

\mathchardef\PhiG="0108
\mathchardef\PsiG="0109
\mathchardef\OmegaG="010A
\newcommand{\demi}{\frac{1}{2}}
\newcommand{\Chat}{{\widehat\C}}

\renewcommand{\L}{{\mathscr L}}
\newcommand{\N}{{\mathord{\mathbb N}}}
\newcommand{\R}{{\mathord{\mathbb R}}}
\newcommand{\Z}{{\mathord{\mathbb Z}}}
\newcommand{\C}{{\mathord{\mathbb C}}}

\newcommand{\E}{{\cal E}}
\renewcommand{\O}{\OmegaG}

\newcommand{\D}{{\cal D}}

\newcommand{\norm}[1]{\|#1\|}

\newcommand{\scalp}[1]{\langle\,#1\,\rangle}

\newcommand{\Real}{{\rm Re\,}}
\newcommand{\Imag}{{\rm Im\,}}

\newcommand{\rmd}{\,\mathrm{d}}
\newcommand{\ds}{\rmd s}

\newcommand{\dx}{\rmd x}

\newcommand{\dE}{\rmd E}

\newcommand{\da}{\rmd \alpha}
\newcommand{\db}{\rmd \beta}
\newcommand{\dc}{\rmd \gamma}
\newcommand{\asymm}{\alpha^{\rm symm}}
\newcommand{\dat}{\rmd \asymm}
\newcommand{\rmi}{\mathrm{i}}
\newcommand{\eps}{\varepsilon}

\newcommand{\trace}{\operatorname{Tr}}

\def\req#1{{\rm(\ref{eq:#1})}}

\makeatletter
\newcommand{\dupdots}{\mathinner{\mkern1mu\raise\p@
    \vbox{\kern7\p@\hbox{.}}\mkern2mu
    \raise4\p@\hbox{.}\mkern2mu\raise7\p@\hbox{.}\mkern1mu}}
\makeatother
\makeatletter
\newcount\c@MaxMatrixCols \c@MaxMatrixCols=10
\newenvironment{cmatrixc}{%
  \hskip -\arraycolsep\array{*\c@MaxMatrixCols c}%
}{%
  \endarray \hskip -\arraycolsep
}
\makeatother
\newenvironment{cmatrix}{\left[\cmatrixc}{\endmatrix\right]}

\newcommand{\rtilde}{{\widetilde{r}}}
\newcommand{\utilde}{{\widetilde{u}}}
\newcommand{\tpsi}{{\widetilde{\psi}}}
\newcommand{\di}{\partial}

\title{Polarization tensors of planar domains as functions
 of the admittivity contrast}
\author{Roland Griesmaier\thanks{Institut f\"ur Mathematik,
    Universit\"at W\"urzburg, 97074 W\"urzburg, Germany
    ({\tt roland.griesmaier@uni-wuerzburg.de})} 
    \and
    Martin Hanke\thanks{Institut f\"ur Mathematik, Johannes
    Gutenberg-Universit\"at Mainz, 55099 Mainz, Germany
    ({\tt hanke@math.uni-mainz.de})}}

\begin{document}
\sloppy
\maketitle

\begin{abstract}
(Electric) polarization tensors describe part of the
leading order term of asymptotic voltage perturbations caused by 
low volume fraction inhomogeneities of the electrical properties of a
medium. 
They depend on the geometry of the support of the inhomogeneities
and on their admittivity contrast.
Corresponding asymptotic formulas are of particular interest in the
design of reconstruction algorithms for determining the locations and
the material properties of inhomogeneities inside a body 
from measurements of current flows and associated voltage
potentials on the body's surface.
In this work we consider the two-dimensional case only and provide an
analytic representation of the polarization tensor in terms of
spectral properties of the double layer integral operator
associated with the support of simply connected conductivity inhomogeneities. 
Furthermore, we establish that an (infinitesimal) simply connected
inhomogeneity has the shape of an ellipse, if and only if the
polarization tensor is a rational function of the admittivity contrast
with at most two poles whose residues satisfy a certain algebraic
constraint. 
We also use the 
analytic representation to provide a
proof of the so-called Hashin-Shtrikman bounds for polarization tensors;
a similar approach has been taken previously by 
Golden and Papanicolaou and Kohn and Milton
in the context of anisotropic composite materials.
\end{abstract}


\begin{AMS}
{\sc 35R30, 65N21}
\end{AMS}

\hspace*{-0.7em}
{\footnotesize \textbf{Last modified.} \today}

\pagestyle{myheadings}
\thispagestyle{plain}
\markboth{R.~GRIESMAIER, M. HANKE}
{POLARIZATION TENSORS}

\addtocounter{footnote}{2}

\section{Introduction}
\label{Sec:Intro}
Electrical impedance tomography is an imaging modality that seeks to
recover the electrical conductivity distribution inside a body from
measurements of current flows and voltage potentials on its surface.
Motivated by numerous potential applications, over the past years a
considerable amount of work has been dedicated to use such measurements
for the reconstruction of low volume fraction conductivity inhomogeneities 
inside a known background medium.
Efficient reconstruction methods for this problem usually rely on
asymptotic representation formulas for the  voltage perturbation
caused by such anomalies (see, e.g., \cite{BHV03}, or Ammari and
Kang~\cite{AmKa07} and the references therein). 
The relevant leading order term in these asymptotic formulas
essentially consists of three parts, namely (i) the gradient of a
certain fundamental solution for the background medium,
(ii) the gradient of the background
potential, and (iii) the (electric) polarization tensor. 
While the first two components are independent of any inhomogeneities,
the polarization tensor fully describes the dependence on the geometry
of the support of the inhomogeneity and on its admittivity (i.e.,
complex conductivity) contrast (see, e.g., Cedio-Fengya, Moskow, and
Vogelius~\cite{CMV98}, or \cite{AmKa07}).

In this work we study analytic properties of the polarization tensor as
function of the admittivity contrast in a restricted setting:
We consider the two-dimensional case only, and we assume that the
conductivity distribution of the background medium as well as the
conductivity inside the inhomogeneity are constant, and that the
support of the conductivity inhomogeneity is a bounded and simply connected
Lipschitz domain. 
Under these assumptions we establish a 
representation formula
for the polarization tensor as function of the admittivity contrast in
terms of the spectral decomposition of the double layer integral
operator associated with the support of the conductivity inhomogeneity. 

This analytic representation of the polarization tensor proves to be a
useful tool to analyze certain aspects of the impedance imaging
problem with low volume fraction conductivity inhomogeneities.
For instance we utilize it in the second part of
this work to provide a short and elementary proof of the so-called
Hashin-Shtrikman bounds~\cite{HaSh63} for the trace of polarization
tensors and their inverses in terms of the volume of the support of the
associated conductivity inhomogeneity, including a discussion of the
sharpness of these inequalities,
similar to a line of arguments worked out by 
Golden and Papanicolaou~\cite{GoPa83} and Kohn and Milton~\cite{KoMi86}
in the framework of anisotropic composites
(see also Belyaev and
Kozlov~\cite{BeKo93}, Lipton~\cite{Lip93}, or Capdeboscq and
Vogelius~\cite{CapVog06}).
These bounds can, e.g., be used to obtain volume estimates for 
unkown conductivity inhomogeneities from boundary measurements.
We also elaborate on how to extract shape information about the support of an 
unknown simply connected conductivity inhomogeneity from knowledge of the
associated polarization tensor as function of the admittivity contrast.
More precisely, we show that the polarization tensor is a rational
function of the admittivity contrast with at most two poles and a certain
constraint on its residues, if and
only if the support of the conductivity inhomogeneity is an ellipse. 

We note that data for the polarization tensor corresponding to several
admittivity contrasts can, e.g., be acquired using multi-frequency
electrical impedance tomography measurements as described in 
our companion paper \cite{GrHa14b}. 
There the analytic representation formula for the polarization
tensor to be established below is 
employed to provide a comprehensive
theoretical justification of so-called multi-frequency MUSIC-type
methods for electrical impedance tomography. 

This paper is organized as follows.
In the next section we review some facts concerning the mapping
properties of layer potential operators associated with the Laplace
operator for planar Lipschitz domains.
In particular we discuss spectral properties of the double layer
integral operator. 
In Section~\ref{Sec:poltensor} we use these results to derive the
representation formula for polarization tensors as a function of the
admittivity contrast; a more explicit version of this formula,
which holds under additional smoothness assumptions, is discussed
in Section~\ref{Sec:smooth}. 
Thereafter we present the two aforementioned applications of these 
representation formulas:
in Section~\ref{Sec:Riemann} we characterize the polarization tensors
corresponding to ellipses as inhomogeneities, and
in Section~\ref{Sec:inequalities} we present 
the proof of the Hashin-Shtrikman bounds. 
We conclude with some final remarks.

\section{Layer potential operators and their spectra}
\label{Sec:Basics}
Let $\O\subset\R^2$ be a bounded Lipschitz domain, and let $\nu$ be
the unit outward normal of $\O$, which is well-defined a.e.\ on
$\partial\O$. 
In the sequel we review some 
classical potential theory
for such domains in as much as it is needed here; as a general
reference we refer to the monograph by McLean~\cite{McLe00}.

Denoting the fundamental solution of the two-dimensional Laplacian by
\bdm
   \Phi(x,y) \,=\, -\frac{1}{2\pi}\,\log|x-y| \,, \qquad 
   x,y \in \R^2\,,\; x\not=y \,,
\edm
the single and the double layer operators associated with $\partial\O$ are
defined by
\bdm
   (S\chi)(x) \,=\, \int_{\partial\O} \Phi(x,y)\chi(y)\ds_y\,, \qquad
   x\in\partial\O\,,
\edm
and
\be{K}
   (K\varphi)(x) 
   \,=\, \int_{\partial\O} \partial_{\nu(y)}\Phi(x,y)\varphi(y)\ds_y\,, \qquad
   x\in\partial\O\,,
\ee
respectively, where for a.e.\ $x\in\partial\O$ the latter is to be
understood in the sense of a Cauchy principal value integral, 
when the boundary of $\O$ lacks smoothness. 
While the single layer operator belongs to
$\L\bigl(H^{-1/2}(\partial\O),H^{1/2}(\partial\O)\bigr)$\footnote{Throughout
  $\L(X,Y)$ denotes the space of bounded linear operators between
  normed spaces $X$ and $Y$, and $\L(X) := \L(X,X)$.},
the double layer operator is a continuous operator in
$\L\bigl(H^{1/2}(\partial \O)\bigr)$.

For $\varphi\in H^{1/2}(\partial\O)$ we further introduce the double layer 
potential
\be{u}
   u(x) \,=\, \int_{\partial\O} \partial_{\nu(y)}\Phi(x,y)\varphi(y)\ds_y\,, 
   \qquad x\in\O\,,
\ee
and the single layer potential 
\be{v}
   v(x) \,=\, \int_{\partial\O} \Phi(x,y) \varphi'(y)\ds_y\,, \qquad
   x \in \O\,,
\ee
with density $\varphi'=\partial_s\varphi$.
Here, 
\bdm
   \partial_s \,:\,
   H^{1/2}(\partial\O)\to H^{-1/2}(\partial\O)
\edm
denotes the derivative operator with respect to arc length 
along $\partial\O$ whose dual operator $\partial_s'$
with respect to the usual sesqui-linear duality pairing $\scalp{\cdot,\cdot}$ 
of $H^{\pm 1/2}(\partial\O)$ satisfies $\partial_s'=-\partial_s$.
The jump relations of single and double layer potentials imply that
the trace of $u$ and the normal derivative of $v$ on $\partial\O$
satisfy 
\bdm
   u|_\Gamma \,=\, \bigl(K-\frac{1}{2}I\bigr)\varphi 
   \qquad \text{and} \qquad
   \partial_\nu v \,=\, \bigl(K'+\frac{1}{2}\bigr)\varphi'\,,
\edm
where $K'\in\L\bigl(H^{-1/2}(\partial\O)\bigr)$ is the dual operator of $K$.

As established by Verchota~\cite[Sect.~4]{Verc84} for Lipschitz
domains, $v$ is 
a harmonic conjugate function of $u$, and hence,
the above jump relations in combination with the Cauchy-Riemann
differential equations imply that
\bdmal
   K'\varphi' 
   &\,=\, \bigl(K'+\frac{1}{2}I\bigr)\varphi' \,-\, \frac{1}{2}\,\varphi'
    \,=\, \partial_\nu v \,-\, \frac{1}{2}\,\varphi'
    \,=\, -\partial_s u|_{\partial\O} \,-\, \frac{1}{2}\,\varphi'\\[1ex]
   &\,=\, -\partial_s\bigl(K-\frac{1}{2}I) \varphi \,-\, \frac{1}{2}\,\varphi'
    \,=\, -\partial_s K\varphi\,.
\edmal
Likewise we conclude that the hypersingular operator 
$T\in\L\bigl(H^{1/2}(\partial\O),H^{-1/2}(\partial\O)\bigr)$ 
which maps $\varphi \in H^{1/2}(\partial\O)$ onto the normal
derivative of the double layer potential
\req{u} satisfies
\bdm
   T\varphi \,=\, \partial_\nu u \,=\, \partial_s v
   \,=\, \partial_s S \varphi'\,.
\edm
From these observations we conclude that
\be{ds-durchschieben}
   K'\partial_s \,=\, -\partial_s K \qquad \text{and} \qquad 
   T \,=\, \partial_s S\partial_s\,,
\ee
which can subsequently be used to rewrite the fundamental 
Plemelj identities (\cite[p.~244]{McLe00}) in the form
\be{Plemelj}
   KS\,=\,SK'
   \qquad \text{and} \qquad
   (S\partial_s)^2 \,=\, K^2-\frac{1}{4}I\,.
\ee

Now we turn to the spectral properties of the double layer operator.
For smooth domains the spectrum of $K$ is fairly well understood since
the work of Plemelj~\cite{Plem11}.
In this case the operator $K$ is compact, so that its spectrum
consists of the origin and a point spectrum, sometimes referred to as
Fredholm eigenvalues of the domain $\O$. 
We refer to Schiffer~\cite{Schi11} or Khavinson, Putinar, and
Shapiro~\cite{KPS07} for expositions of the corresponding results. 
These have recently been extended by Helsing and Perfekt~\cite{HePe13}
to Lipschitz domains, where $K$ will no longer be compact, in
general. 
While the focus in \cite{HePe13} is primarily on the three-dimensional
case, we will develop the corresponding two-dimensional theory below.

At this point we make the additional assumption that $\O$ be simply
connected and that $\O\subset\D$, where $\D$ denotes the unit disk. 
The first of these two assumptions implies that the trivial eigenvalue
$\lambda=-1/2$ of $K$ has a one-dimensional eigenspace, which is
spanned by the constant functions on $\partial\O$. 
The second assumption, which we remove below when we formulate (and
prove) our main result (Theorem~\ref{Thm:main}), makes sure that the
single layer operator $S$ is selfadjoint and positive as operator in
$\L\bigl(L^2(\partial\O)\bigr)$ (cf., e.g., Landkof~\cite{Land72}),
and is an isomorphism between $H^{-1/2}(\partial\O)$ and 
$H^{1/2}(\partial\O)$.
As such, $S$ has a well-defined square root operator $S^{1/2}$ that
is an isomorphism in
$\L\bigl(H^{-1/2}(\partial\O),L^2(\partial\O)\bigr)$ and in 
$\L\bigl(L^2(\partial\O),H^{1/2}(\partial\O)\bigr)$.

Now we introduce
\be{A}
   A \,=\, S^{-1/2}KS^{1/2}\,,
\ee
which is a selfadjoint operator in $L^2(\partial\O)$ by virtue of
\req{Plemelj}. 
Note that \req{A} is a similarity transformation, hence $A$ and $K$
have the same eigenvalues and spectra, and associated eigenspaces have
the same dimensions, respectively. In particular, $\lambda=-1/2$ is an
eigenvalue of $A$, and its eigenspace is spanned by $S^{-1/2}1$.
On the orthogonal complement of this eigenspace the shifted operators 
$A\pm \demi I$ are contracting (cf., e.g., Steinbach and
Wendland~\cite{StWe01}). 
Hence we conclude that the spectrum $\sigma(A)$ of $A$ is contained in
\be{sigma-A}
   \sigma(A) \,\subset\, \{-1/2\} \cup [-a,a]
\ee
for some $0\leq a<1/2$. Below, we denote by $E$ the 
spectral decomposition of $A$, i.e., 
for every Borel set $\omega\subset\R$ there is an orthogonal
projection  $E(\omega) \in \L\bigl(L^2(\partial \O)\bigr)$ such that
\bdm
   A \,=\, \int_{-\infty}^\infty \lambda \dE_\lambda
\edm
(cf., e.g., Rudin~\cite[Sect.~12.17]{Rud73}). 

For smooth domains Plemelj~\cite{Plem11} has observed that in two
space dimensions the spectra of $A$ and $K$ have rich additional
structure due to complex function theory: 
Namely, the first identity in \req{ds-durchschieben} implies that if 
$\varphi\in H^{1/2}(\partial\O)$ is an eigenfunction of $K$ for the
eigenvalue $\lambda\neq-1/2$ then $\varphi'$ is an eigenfunction of
$K'$ for $-\lambda$. 
(If $\varphi$ is an eigenfunction of $K$ for $\lambda=-1/2$ 
then $\varphi$ is constant and $\varphi'$ vanishes.) 
Since the spectra of $K$ and $K'$ are the same it follows that the
eigenvalues of $K$ different from $\lambda=-1/2$ come in pairs that
are symmetric with respect to the origin. 
Concerning the operator $A$ we note that \req{ds-durchschieben} and
\req{Plemelj} imply that 
\be{durchschieben}
   AS^{1/2}\partial_s S^{1/2}
   \,=\, S^{1/2}K'\partial_s S^{1/2}
   \,=\, -S^{1/2}\partial_s KS^{1/2}
   \,=\, -S^{1/2}\partial_s S^{1/2}A \,,
\ee
and hence, if $\psi$ is an eigenfunction of $A$ with eigenvalue 
$\lambda\neq-1/2$
then $S^{1/2}\partial_s S^{1/2}\psi$ is an eigenfunction of $A$ for the
eigenvalue $-\lambda$.

The analog of Plemelj's result for Lipschitz domains---where $A$ need not
have any eigenfunctions besides $S^{-1/2}1$---is the following 
result.

\begin{proposition}
\label{Prop:Plemelj}
Assume that $\O\subset\D$ is a simply connected Lipschitz domain. 
Then for every Borel set $\omega\subset\R$ there holds
\bdm
  E(\omega)S^{1/2}\partial_s S^{1/2}
  \,=\, S^{1/2}\partial_s S^{1/2} E(-\omega)\,.
\edm
\end{proposition}

\begin{proof}
Starting from \req{durchschieben} we readily obtain by induction that
\bdm
   A^nS^{1/2}\partial_s S^{1/2}
   \,=\, S^{1/2}\partial_s S^{1/2}(-A)^n \qquad 
   \text{for every } n\geq 0 \,.
\edm
Since
\begin{equation*}
  \langle A^n(S^{1/2}\partial_s S^{1/2})\psi,\tpsi\rangle
  \,=\, \int_{-\infty}^\infty \lambda^n \rmd 
  \langle E_\lambda(S^{1/2}\partial_s S^{1/2})\psi,\tpsi\rangle
\end{equation*} 
for every $\psi,\tpsi \in L^2(\di\O)$, and similarly
\begin{equation*}
  \begin{split}
    \langle &(S^{1/2}\partial_s S^{1/2})(-A)^n\psi,\tpsi\rangle
     \,=\, \langle (-A)^n \psi,(S^{1/2}\partial_s S^{1/2})^*\tpsi\rangle\\
    &\,=\, \int_{-\infty}^\infty \!\! (-\lambda)^n \rmd \langle E_\lambda\psi, 
    (S^{1/2}\partial_s S^{1/2})^*\tpsi\rangle
     \,=\, \int_{-\infty}^\infty \!\! \lambda^n \rmd \langle E_{-\lambda}\psi,
    (S^{1/2}\partial_s S^{1/2})^*\tpsi\rangle\\
    &\,=\, \int_{-\infty}^\infty \!\!\lambda^n \rmd 
    \langle (S^{1/2}\partial_s S^{1/2}) E_{-\lambda}\psi,\tpsi\rangle \,,
  \end{split}
\end{equation*}
we find that the measures 
$\langle E(\omega)S^{1/2}\partial_s S^{1/2}\psi,\tpsi\rangle$ and 
$\langle S^{1/2}\partial_s S^{1/2}) E(-\omega)\psi,\tpsi\rangle$,
considered as linear forms on $C(\R)$, agree on the space of
polynomials. 
Since $E$ is concentrated on a compact interval,
Weierstrass' approximation theorem and Riesz'
representation theorem imply that 
\bdm
  E(\omega)S^{1/2}\partial_s S^{1/2}
  \,=\, S^{1/2}\partial_s S^{1/2} E(-\omega)
\edm
for every Borel set $\omega\subset\R$, which was to be shown.
\hfill
\end{proof}

\begin{corollary}
\label{Kor:Plemelj}
The set $\sigma(A)\setminus\{-1/2\}$ is a compact subset of $(-1/2,1/2)$
that is symmetric with respect to the origin.
\end{corollary}

\begin{proof}
From \req{sigma-A} follows that $\sigma(A)\setminus\{-1/2\}$ is a compact
subset of $(-1/2,1/2)$.

Now consider any $\lambda \in \sigma(A)\setminus\{-1/2\}$ and
any open neighborhood $\omega\subset\R$ of $-\lambda$. 
Without loss of generality we may assume that $1/2\notin\omega$. 
Because of $\lambda\in-\omega$ we have $E(-\omega)\neq 0$.
Hence there exists
$\psi\in L^2(\partial\O)$ with $\norm{\psi}=1$ and $E(-\omega)\psi=\psi$.
It thus follows from Proposition~\ref{Prop:Plemelj} that
\bdm
   E(\omega)S^{1/2}\partial_sS^{1/2}\psi
   \,=\, S^{1/2}\partial_s S^{1/2}E(-\omega)\psi
   \,=\, S^{1/2}\partial_s S^{1/2} \psi \,,
\edm
and the right-hand side vanishes, if and only if $\psi=cS^{-1/2}1$ for some
$c\in\R$, i.e.,
if and only if $\psi$ is an eigenfunction of $A$ corresponding to $-1/2$.
The latter, however, cannot be true because the range of $E(-\omega)$ is
orthogonal to this eigenspace due to the fact that $-1/2\notin-\omega$.
Thus we have established that $E(\omega)\neq 0$ for every sufficiently
small neighborhood of $-\lambda$. This proves that $-\lambda\in\sigma(A)$, 
which was to be shown.
\hfill
\end{proof}

\section{The polarization tensor as a function of the admittivity contrast}
\label{Sec:poltensor}
We now turn to investigate the polarization tensor of $\O$ and derive
the main result of this work. Again, for the time being we restrict our 
attention to Lipschitz domains $\O\subset\D$.

Let $x_{1,2}$ and $\nu_{1,2}$ denote the two components of the spatial
variable $x\in\R^2$ and of the outer unit normal at $x\in\partial\O$,
respectively. 
Then the quantities
\be{poltensor}
   M_{kl}(\mu;\O) \,=\, \int_{\partial \O} \nu_l (\mu I- K)^{-1} x_k\ds\,, 
   \qquad k,l=1,2\,,
\ee
are well-defined for every $\mu\in\C\setminus\sigma(K)$. They
constitute the entries of the (electric) polarization tensor 
$M(\mu; \O)=[M_{kl}(\mu;\O)]_{kl}\in\C^{2\times2}$ 
associated with the planar domain $\O$
(see \cite[Sect.~4.1]{AmKa07} for the interpretation of $\mu$ as an
admittivity contrast).

Note that 
\bdm
   \int_{\partial \O} \nu_l\ds \,=\, 0\,, \qquad l=1,2\,,
\edm
because $\nu_l = \nu\cdot\grad x_l$, and $x\mapsto x_l$ is harmonic in
$\O$. 
Furthermore, since $K1=-1/2$, any constant shift of $\O$ becomes a
constant shift of $(\mu I- K)^{-1}x_k$, that cancels in the
computation of the polarization tensor. 
It follows that the polarization tensor is independent of translations
of $\O$. 
Finally, given an orthogonal transform $Q \in \R^{2\times 2}$ and
a scaling factor $c>0$, the polarization tensor satisfies
\begin{equation}
\label{eq:similarity}
  M(\,\cdot\,;cQ(\O)) \,=\, c^2 Q M(\,\cdot\,;\O) Q^* 
\end{equation}
(cf., e.g., \cite{AmKa07}). 

We rewrite $x_1$ as a double layer potential
\bdm
   x_1 \,=\, \int_{\partial\O} \partial_{\nu(y)}\Phi(x,y)\varphi(y)\ds_y\,, 
   \qquad
   x\in\O\,,
\edm
with density $\varphi\in H^{1/2}(\partial\O)$, which solves
the second kind integral equation
\be{x1}
   \bigl(K-\frac{1}{2}I\bigr)\varphi \,=\, x_1|_{\partial\O}\,.
\ee
From \req{ds-durchschieben} we conclude that
\bdm
   \nu_1|_{\partial\O} \,=\, \partial_\nu x_1 \,=\, T\varphi 
   \,=\, \partial_s S\partial_s\varphi\,,
\edm
and hence, Plemelj's identities~\req{Plemelj} yield
\be{nu1}
   \nu_1|_{\partial\O} 
   \,=\, S^{-1}(S\partial_s)^2\varphi
   \,=\, S^{-1}\bigl(K^2-\frac{1}{4}I\bigr)\varphi
   \,=\, \bigl({K'}^2-\frac{1}{4}I\bigr)S^{-1}\varphi\,,
\ee
so that
\bdm
\begin{split}
   M_{11}(\mu;\O)
   &\,=\, \big\langle 
             S^{-1}\varphi,
             \bigl(\mu I-K\bigr)^{-1}\bigl(K^2-\frac{1}{4}I\bigr)
             \bigl(K-\frac{1}{2}I\bigr)\varphi
          \big\rangle \\[1ex]
   &\,=\, \big\langle 
             S^{-1/2}\varphi,
             \bigl(\mu I-A\bigr)^{-1}\bigl(A^2-\frac{1}{4}I\bigr)
             \bigl(A-\frac{1}{2}I\bigr)S^{-1/2}\varphi
          \big\rangle\,.
\end{split}
\edm
Using spectral calculus this can be rewritten as
\be{M11}
   M_{11}(\mu;\O)
   \,=\, \int_{-\infty}^\infty \frac{1}{\mu-\lambda} \,\da_\lambda\,,
\ee
where the positive Radon measure $\alpha$ is given by
\be{alpha}
   \alpha(\omega) \,=\, 
   \Bigl\|\,E(\omega)
        \Bigl(\bigl(\frac{1}{4}I-A^2\bigr)\bigl(\frac{1}{2}I-A\bigr)\Bigr)^{1/2}
        S^{-1/2}\varphi\,\Bigr\|^2
\ee
for every Borel set $\omega\subset\R$.
Formula~\req{M11} appears in \cite{HePe13}
(see also Ammari, Chow, Liu, and Zou~\cite{ACLZ13}), and
similar representations of effective conductivity tensors of
composite media go back to Bergman~\cite{Berg78}; see also
\cite{GoPa83,KoMi86} and the corresponding Chapter~XVIII of the
monograph by Milton~\cite{Milt02}.

The harmonic conjugate $v=x_2$ of $u=x_1$ in $\O$ is 
specified by \req{v} up to an additive constant, where
$\varphi'=\partial_s\varphi$ again. 
Accordingly,
\be{nu2}
   \nu_2|_{\partial\O} \,=\, \partial_\nu x_2
   \,=\, \bigl(K'+\frac{1}{2}I\bigr)\varphi'\,,
\ee
and hence, 
\be{M22help}
\begin{split}
   M_{22}(\mu;\O) 
    &\,=\, \big\langle 
             \bigl(K'+\frac{1}{2}I\bigr)\varphi',(\mu I-K)^{-1}S\varphi'
          \big\rangle \\[1ex]
   &\,=\, \big\langle 
             S^{1/2}\varphi',\bigl(A+\frac{1}{2}I\bigr)(\mu I-A)^{-1}
             S^{1/2}\varphi'
          \big\rangle \\[1ex]
   &\,=\, \int_{-\infty}^\infty \frac{1}{\mu-\lambda} \,\db_\lambda
\end{split}
\ee
with
\be{beta}
   \beta(\omega) \,=\, 
   \Bigl\|\,E(\omega)\bigl(\frac{1}{2}I+A\bigr)^{1/2}
            S^{1/2}\varphi'\,\Bigr\|^2
\ee
for every Borel set $\omega\subset\R$.
As has been pointed out by Keller~\cite{Kell64} the usage of these
harmonic conjugates reveals duality properties of these 
two-dimensional tensors. More precisely, we have the following result.

\begin{lemma}
\label{Lem:beta}
Let $\O\subset\D$ be a simply connected Lipschitz domain. Then
\bdm
   \beta(\omega) \,=\, \alpha(-\omega)
\edm
for every Borel set $\omega\subset\R$.
\end{lemma}

\begin{proof}
For a given $\omega$ we use Proposition~\ref{Prop:Plemelj} and the fact
that every function of $A$ commutes with (the projection) $E(\omega)$
to rewrite
\bdmal
   \beta(\omega) 
   &\,=\, \bigl\langle E(\omega)S^{1/2}\varphi',
                 \bigl(\frac{1}{2}I+A\bigr)E(\omega)S^{1/2}\varphi'\bigr\rangle\\
   &\,=\, \bigl\langle S^{1/2}\partial_s S^{1/2}E(-\omega)S^{-1/2}\varphi,
                 \bigl(\frac{1}{2}I+A\bigr)S^{1/2}\partial_s S^{1/2}
                 E(-\omega)S^{-1/2}\varphi \bigr\rangle \,.
\edmal
Since $S^{1/2}$ is selfadjoint and $\partial_s$ is skewadjoint
it thus follows from \req{durchschieben}, \req{ds-durchschieben},
\req{Plemelj}, and \req{A} that 
\bdmal
   \beta(\omega) 
   &\,=\, -\bigl\langle E(-\omega)S^{-1/2}\varphi,
          S^{1/2}\partial_s S^{1/2}\bigl(\frac{1}{2}I+A\bigr)
          S^{1/2}\partial_s S^{1/2}E(-\omega)S^{-1/2}\varphi \bigr\rangle\\
   &\,=\, -\bigl\langle E(-\omega)S^{-1/2}\varphi,
                 \bigl(\frac{1}{2}I-A\bigr)S^{-1/2}(S\partial_s)^2 S^{1/2} 
                 E(-\omega)S^{-1/2}\varphi \bigr\rangle\\
   &\,=\, -\bigl\langle E(-\omega)S^{-1/2}\varphi,
                 \bigl(\frac{1}{2}I-A\bigr)\bigl(A^2-\frac{1}{4} I\bigr)
                 E(-\omega)S^{-1/2}\varphi \bigr\rangle \,.
\edmal
A comparison with \req{alpha} now yields the assertion.
\hfill
\end{proof}

Making use of this result we can rewrite \req{M22help} to obtain the
formula 
\be{M22M11}
   M_{22}(\mu;\O) 
   \,=\, \int_{-\infty}^\infty \frac{1}{\mu+\lambda} \,\da_\lambda
   \,=\, -M_{11}(-\mu;\O),
\ee
see \cite{Kell64} again.
Next we consider the off-diagonal entries of the polarization tensor.
Using \req{x1} and \req{nu2} there holds
\be{M12-tmp}
\begin{aligned}
   M_{12}(\mu;\O)
   &\,=\, \big\langle
             \bigl(K'+\frac{1}{2}I\bigr)\varphi',
             \bigl(\mu I-K\bigr)^{-1}\bigl(K-\frac{1}{2}I\bigr)\varphi
          \big\rangle \\[1ex]
   &\,=\, \big\langle 
             S^{1/2}\varphi',
             \bigl(\mu I-A\bigr)^{-1}\bigl(A^2-\frac{1}{4}I\bigr)
             S^{-1/2}\varphi
          \big\rangle\,,
\end{aligned}
\ee
and hence,
\be{M12}
   M_{12}(\mu;\O)
   \,=\, \int_{-\infty}^\infty \frac{1}{\mu-\lambda}\,\dc_\lambda
\ee
with
\be{gamma}
   \gamma(\omega) \,=\, 
   \bigl\langle S^{1/2}\varphi',
   E(\omega)\bigl(A^2-\frac{1}{4}I\bigr)S^{-1/2}\varphi \bigr\rangle
\ee
for every Borel set $\omega\subset\R$.
Employing \req{nu1} it is straightforward to deduce the same 
representation for $M_{21}(\mu;\O)$. We mention, though, that
it is well-known
that the polarization tensor is symmetric (see, e.g., \cite{AmKa07}).

A similar argument as in Lemma~\ref{Lem:beta} reveals that the measure $\gamma$
is antisymmetric:

\begin{lemma}
\label{Lem:gamma}
Let $\O\subset\D$ be a simply connected Lipschitz domain. Then
\bdm
   \gamma(\omega) \,=\, -\gamma(-\omega)
\edm
for every Borel set $\omega\subset\R$.
Moreover, 
\be{estgamma}
   |\gamma(\omega)| 
   \,\leq\, \bigl(\alpha(\omega)\alpha(-\omega)\bigr)^{1/2}\!
   \,\leq\, \asymm(\omega) 
   \,:=\, \frac{1}{2}\bigl(\alpha(\omega) + \alpha(-\omega)\bigr)\,.
\ee
\end{lemma}

\begin{proof}
Using Proposition~\ref{Prop:Plemelj}, \req{durchschieben}, and
\req{gamma}, we find that for any Borel set $\omega\subset\R$ there
holds 
\bdmal
   \gamma(\omega) 
   &\,=\, \bigl\langle E(\omega)S^{1/2}\varphi',
                 \bigl(A^2-\frac{1}{4}I\bigr)E(\omega)S^{-1/2}\varphi \bigr\rangle\\
   &\,=\, \bigl\langle S^{1/2}\partial_s S^{1/2}E(-\omega)S^{-1/2}\varphi,
                 \bigl(A^2-\frac{1}{4}\bigr)E(\omega)S^{-1/2}\varphi \bigr\rangle\\
   &\,=\, -\bigl\langle E(-\omega)S^{-1/2}\varphi,
                 S^{1/2}\partial_s S^{1/2}\bigl(A^2-\frac{1}{4}I\bigr)
                 E(\omega)S^{-1/2}\varphi \bigr\rangle\\
   &\,=\, -\bigl\langle E(-\omega)S^{-1/2}\varphi,
                 \bigl(A^2-\frac{1}{4}I\bigr)S^{1/2}\partial_s S^{1/2} 
                 E(\omega)S^{-1/2}\varphi \bigr\rangle\\
   &\,=\, -\bigl\langle E(-\omega)S^{-1/2}\varphi,
                 \bigl(A^2-\frac{1}{4}I\bigr)E(-\omega)S^{1/2}\varphi' \bigr\rangle\\
   &\,=\, -\gamma(-\omega) \,.
\edmal
Concerning the second assertion we apply the Cauchy-Schwarz inequality to
estimate
\bdmal
    |\gamma(\omega)| 
    &\,=\, \Bigl|\bigl\langle
    E(\omega)\bigl(\frac{1}{2}I+A\bigr)^{1/2} S^{1/2}\varphi',
    E(\omega)\bigl(\frac{1}{2}I+A\bigr)^{1/2}
    \bigl(\frac{1}{2}I-A\bigr)S^{-1/2}\varphi\bigr\rangle\Bigr|\\
    &\,\leq\, \bigl(\alpha(\omega)\beta(\omega)\bigr)^{1/2} \! 
     \,\leq\, \frac12 \bigl(\alpha(\omega)+\beta(\omega)\bigr)\,.
\edmal
A final application of Lemma~\ref{Lem:beta} yields the desired 
inequality~\req{estgamma}.
\hfill
\end{proof}

We are now ready to state our main result, for which we drop the
assumption that $\O\subset\D$.

\begin{theorem}
\label{Thm:main}
Let $\O\subset\R^2$ be a bounded simply connected Lipschitz domain. 
Then the polarization tensor $M(\mu;\O)$ from \eqref{eq:poltensor}
is of the form
\bdm
   M(\mu;\O) \,=
   \begin{cmatrix}
   {\displaystyle \phantom{x}
      \int_{-\infty}^\infty 
             \frac{r^2(\lambda)}{\mu-\lambda} \,\dat_\lambda
   }  &
   {\displaystyle
      \int_{-\infty}^\infty 
             \frac{c(\lambda)r(\lambda)r(-\lambda)}{\mu-\lambda} \,
      \dat_\lambda 
      \phantom{x}
   }  \\[4ex]
   {\displaystyle \phantom{x}
      \int_{-\infty}^\infty 
             \frac{c(\lambda)r(\lambda)r(-\lambda)}{\mu-\lambda} \,
      \dat_\lambda
   }  &
   {\displaystyle
      \int_{-\infty}^\infty
             \frac{r^2(-\lambda)}{\mu-\lambda} \,\dat_\lambda  \phantom{x}
   }
   \end{cmatrix}
\edm
for $\mu\in\C\setminus\sigma(K)$, 
where $\asymm$ is a positive symmetric Radon measure supported on 
$\sigma(K)\setminus\{-1/2\}$, which is a compact subset of $(-1/2,1/2)$,
$r$ is a nonnegative $L^\infty(\R;\asymm)$ function and $c$ is an 
odd\footnote[1]{Here, odd means that $c(\lambda)=-c(-\lambda)$ 
for $\asymm$ a.e.\ $\lambda\in\R$; 
in particular, $c$ can be chosen to satisfy $c(0)=0$.}
and real valued $L^\infty(\R;\asymm)$ function with \mbox{$|c| \leq 1$.}
Moreover, there holds
\be{Mkklimit}
   \int_{-\infty}^\infty r^2(\lambda) \dat_\lambda
   \,=\, \lim_{\mu\to\infty} \mu M_{kk}(\mu;\O) 
   \,=\, |\O|
\ee
for $k=1,2$.
\end{theorem}

\begin{proof}
Assume first that $\O\subset\D$, so that the previous results
apply, and define $\alpha$, $\beta$, and $\gamma$ as in \req{alpha},
\req{beta}, and \req{gamma}, respectively, where 
$\varphi\in H^{1/2}(\partial\O)$ is the unique solution of \req{x1}.
Then, by virtue of Lemma~\ref{Lem:gamma} there holds $\gamma \ll \asymm$, 
while $\alpha \ll \asymm$ is a trivial statement.
Accordingly the Radon-Nikodym Theorem (cf., e.g., Evans and
Gariepy~\cite[p.~40]{EvaGar92}) implies the existence of functions 
$h,\, \rtilde \in L^1(\R;\asymm)$ such that 
\bdm
  \gamma(\omega) \,=\, \int_\omega h(\lambda) \dat_\lambda \qquad
  \text{and} \qquad
  \alpha(\omega) \,=\, \int_\omega \rtilde(\lambda) \dat_\lambda 
\edm
for every Borel set $\omega\subset\R$.
A localization argument (cf., e.g., \cite[p.~43]{EvaGar92}) shows that
we can choose $\rtilde$ to satisfy
\bdm
   0 \,\leq\, \rtilde(\lambda) \,\leq\, 2\,, \qquad \lambda\in\R\,,
\edm 
and hence we can rewrite $\rtilde = r^2$ for some nonnegative 
$r \in L^\infty(\R;\asymm)$.
Similarly, we find from Lemma~\ref{Lem:gamma} that 
$h\in L^\infty(\R;\asymm)$ can be chosen to satisfy 
$h(\lambda)=-h(-\lambda)$ for $\asymm$ a.e.\ $\lambda\in\R$, and
\bdm
   |h(\lambda)|\,\leq\, \bigl(\rtilde(\lambda)\rtilde(-\lambda)\bigr)^{1/2}
   \,=\, r(\lambda)r(-\lambda)\,, \qquad 
   \lambda\in\R\,.
\edm
Thus the odd function
\begin{equation*}
  c(\lambda) \,:=\,
  \begin{cases}
    \dfrac{h(\lambda)}{r(\lambda)r(-\lambda)} 
       &\text{if } r(\lambda)r(-\lambda)\not=0\,,\\[1ex]
    \phantom{xxxx}1 &\text{else}\,,
  \end{cases}
\end{equation*}
is Borel measurable on $\R$ and satisfies $|c| \leq 1$.
Now the asserted representation of $M(\mu;\O)$
follows from \req{M11}, \req{M22M11}, and \req{M12}.

Concerning the support of $\asymm$ we first recall from 
Corollary~\ref{Kor:Plemelj} that 
$\sigma(K)\setminus\{-1/2\}=\sigma(A)\setminus\{-1/2\}$ 
is a compact subset of $(-1/2,1/2)$ that is symmetric with respect
to the origin. Moreover, from
\bdm
   \scalp{S^{1/2}\varphi',S^{-1/2}1} \,=\, \scalp{\varphi',1} \,=\, 0
\edm
we conclude that $S^{1/2}\varphi'$ is orthogonal to the eigenspace of $A$
corresponding to the eigenvalue $\lambda=-1/2$. 
Hence, it follows from \req{beta} that $\beta$ is supported in
$\sigma(A)\setminus\{-1/2\}$, and the same must be true for $\asymm$
due to Lemma~\ref{Lem:beta} and the definition of $\asymm$ in
\req{estgamma}. This proves the representation formula for $M(\mu;\O)$
under the additional assumption that $\O\subset\D$.

For a general domain we choose $\eps>0$ sufficiently small so that
$\eps\O\subset\D$. Making use of \req{similarity}, the above
result (with measure $\asymm)$ for $\eps\O$ provides the corresponding 
result for $\O$, with $\asymm$ replaced by $\asymm/\eps^2$. 
Note that by virtue of \req{K} the spectrum of $K$ is independent of 
rescalings of the domain.

Finally, while the first equality in \req{Mkklimit} is an immediate
consequence of the representation of the diagonal entries of the
polarization tensor, the original definition~\req{poltensor} of 
$M_{kk}(\mu;\O)$ and Green's formula yield
\bdmal
   \mu M_{kk}(\mu;\O) 
   &\,=\, \int_{\partial\O} \nu_k\bigl(I-\frac{1}{\mu} K\bigr)^{-1}\!x_k\ds
          \\[1ex]
   &\!\longrightarrow\, \int_{\partial\O} \nu_k x_k\ds
    \,=\, \int_\O |\nabla x_k|^2 \dx 
    \,=\, |\O|
\edmal
for $k=1,2$ as $\mu\to\infty$, and hence the second identity in
\req{Mkklimit} follows. 
\hfill
\end{proof}

We mention that \req{Mkklimit} is well known (cf., e.g., \cite{HePe13}).
\begin{remark}
\rm
According to Theorem~\ref{Thm:main} $M(\mu;\O)$ 
extends as an analytic function to $\C\setminus[-a,a]$ with $0\leq a<1/2$
as in \req{sigma-A}, and this extension coincides with 
the original definition~\req{poltensor} for $\mu\neq -1/2$.
The extension to $\mu=-1/2$ satisfies the alternative (dual) representation 
\bdm
   M_{kl}(\mu;\O) \,=\, 
   \int_{\partial \O} x_k (\mu I- K_\diamond')^{-1} \nu_l\ds\,, 
   \qquad k,l=1,2\,,
\edm
in which $K'_\diamond$ denotes the restriction of $K'$ to its 
invariant subspace
\bdm
   L^2_\diamond(\partial\O) \,=\,
   \{\, \chi\in L^2(\partial\O)\,:\, \int_{\partial\O}\chi\ds = 0\,\}\,,
\edm
on which $\demi I-K_\diamond'$ is invertible (cf., e.g., \cite{AmKa07}).
\fin
\end{remark}

\section{The polarization tensor of smooth planar domains}
\label{Sec:smooth}
When the double layer operator 
$K:H^{1/2}(\partial\O)\to H^{1/2}(\partial\O)$ is compact,
for example, if $\O$ is a $C^2$ domain (cf.~Kress~\cite{Kres14}),
the representation of Theorem~\ref{Thm:main} of the polarization
tensor simplifies to 
\be{MlO} 
   M(\mu;\O) \,=
   \begin{cmatrix}
   {\displaystyle \phantom{x}
      \sum_{n\in\Z} \frac{r_n^2}{\mu-\lambda_n} \quad
   }  &
   {\displaystyle
      \sum_{n\in\Z} \frac{c_nr_nr_{-n}}{\mu-\lambda_n}
      \phantom{x}
   }  \\[4ex]
   {\displaystyle \phantom{x}
      \sum_{n\in\Z} \frac{c_nr_nr_{-n}}{\mu-\lambda_n}
      \quad
   }  &
   {\displaystyle
      \sum_{n\in\Z} \frac{r_{-n}^2}{\mu-\lambda_n} 
      \phantom{x}
   }
   \end{cmatrix},
\ee
where $\{\lambda_n\,:\,n\in\N\}$ is an enumeration of the 
nonnegative eigenvalues of $K$ (ignoring multiplicities),
$\lambda_{-n}=-\lambda_n$ for $n\in\N$, and $\lambda_0=0$.
In the notation of Theorem~\ref{Thm:main} the nonnegative coefficients 
$r_{\pm n}$ and the real coefficients $c_n=-c_{-n}$ with $|c_n|\leq 1$ 
are given by
\bdm
   r_n = r(\lambda_n)\asymm(\{\lambda_n\}) \qquad \text{and} \qquad
   c_n = c(\lambda_n) 
\edm
for $n\in\Z$. 
Note that if zero does not happen to be an eigenvalue of $A$
then there holds $r_0=0$; note further that $c_0=0$, independent of
whether $\lambda=0$ is an eigenvalue of $A$, or not. 

We conclude that in the smooth case the polarization tensor 
$M(\mu;\O)$ is a meromorphic function of $\mu\in\Chat\setminus\{0\}$
into the space of complex symmetric $2\times 2$ matrices, with simple
poles at eigenvalues $\lambda\notin\{-1/2,0\}$ of the double layer
operator $K$, and with limiting value zero at infinity.
Here, as usual, $\Chat=\C\cup\{\infty\}$ denotes the extended complex
plane. 

Of course, \req{MlO} is much more easy to derive in the smooth case by
expanding
\bdm
   S^{-1/2}\varphi \,=\, \sum_{n\in\Z} r_n\psi_n\,,
\edm
where $\psi_n\in L^2(\partial\O)$, $n\in\Z$, are eigenfunctions of $A$
for the eigenvalues $\lambda_n$, respectively, normalized to satisfy 
\bdm
   \norm{\psi_n}^2 \,=\, 
   \Bigl(\bigl(\frac{1}{2}-\lambda_n\bigr)\bigl(\frac{1}{4}-\lambda_n^2\bigr)\Bigr)^{-1}, \qquad
   n\in\Z\,.
\edm
Then there holds
\bdm
   S^{1/2}\varphi' \,=\, (S^{1/2}\partial_s S^{1/2})S^{-1/2}\varphi
   \,=\, \sum_{n\in\Z} r_n\tpsi_n\,,
\edm
where $\tpsi_n = S^{1/2}\partial_s S^{1/2} \psi_n  \in L^2(\partial\O)$ 
are eigenfunctions of $A$ corresponding to $-\lambda_n=\lambda_{-n}$ 
(cf.~\req{durchschieben}), and
\bdmal
   \norm{\tpsi_n}^2 
   &\,=\, \norm{S^{1/2}\partial_s S^{1/2} \psi_n}^2
    \,=\, -\scalp{\psi_n,S^{-1/2}(S\partial_s)^2 S^{1/2}\psi_n} \\
   &\,=\, -\bigl\langle \psi_n,\bigl(A^2-\frac{1}{4}I\bigr)\psi_n\bigr\rangle
    \,=\, \bigl(\frac{1}{4}-\lambda_n^2\bigr)\norm{\psi_n}^2 
    \,=\, \bigl(\frac{1}{2}-\lambda_n\bigr)^{-1}.
\edmal

This sheds some additional light on the parameters $c_n$ in \req{MlO}:
From \req{M12-tmp} and the fact that eigenspaces of $A$
corresponding to different eigenvalues are mutually orthogonal we
deduce that 
\bdmal
   M_{12}(\mu;\O) 
   &\,=\, \sum_{n\in\Z} r_{-n}r_n \frac{1}{\mu-\lambda_n} 
                      \bigl(\lambda_n^2-\frac{1}{4}\bigr) \scalp{\tpsi_{-n},\psi_n} \\
   &\,=\, - \sum_{n\in\Z} \frac{r_nr_{-n}}{\mu-\lambda_n}
                    \frac{\scalp{\tpsi_{-n},\psi_n}}
                   {\norm{\tpsi_{-n}}\norm{\psi_n}\vphantom{\dfrac{;}{2}}}\,.
\edmal
It follows that
\be{c-n}
   c_n \,=\, - \frac{\scalp{\tpsi_{-n},\psi_n}}
                    {\norm{\tpsi_{-n}}\norm{\psi_n}\vphantom{\dfrac{;}{2}}}
\ee
is the cosine of the angle between $\psi_n$ and $-\tpsi_{-n}$, both of
which are eigenfunctions of $A$ corresponding to the same eigenvalue
$\lambda_n$. 
In particular, if all eigenspaces of $A$ have dimension one then 
$c_n=\pm 1$ for every $n\in\Z$. 
In general, however, the coefficients $c_n$ can take values different
from $\pm 1$ as the following example shows. 

\begin{example}
\label{Ex:cyclic}
\rm
A domain $\O\subset\R^2$ shall be called 
\emph{cyclic with index $k\geq 2$}, if $Q(\O)=\O$ where $Q$ denotes
the rotation by $2\pi/k$ (clockwise, or counter-clockwise). By virtue
of \req{similarity} there holds 
\bdm
   M(\mu;\O) \,=\, QM(\mu;\O)Q^*
\edm
for every cyclic domain $\O$, and hence, if $p$ is an eigenvector of 
the polarization tensor for the eigenvalue $\zeta$ then $Qp$ is another
eigenvector for the same eigenvalue. Since $p$ and $Qp$ are linearly
independent whenever $k\geq 3$, it thus follows that 
\be{M-cyclic}
   M(\mu;\O) \,=\, f(\mu) I\,,
\ee
where $f$ is a scalar complex function of $\mu$, and 
$I$ is the $2\times2$ identity matrix. 
Compare, e.g., \cite[p.~102]{AmKa07}, for the same argument. 

In particular, if $\O$ is a $C^2$ domain that is cyclic with index 
$k\geq 3$ then 
\bdm
   f(\mu) \,=\, \sum_{n\in\Z} \frac{r_n^2}{\mu-\lambda_n}
\edm
is a scalar meromorphic function of $\mu\in\Chat\setminus\{0\}$
by virtue of \req{MlO}. A further comparison with \req{MlO} yields that
in this case we must have $r_n=r_{-n}$, and hence, that $c_n=0$ for
every pole $\lambda_n$ occuring in the representation formula. 

Thus, it follows from \req{c-n} and the discussion following it that
every Fredholm eigenvalue $\lambda_n$ that is present in \req{MlO}
must have multiplicity greater than one, and that the two
eigenfunctions $\psi_n$ and $\tpsi_{-n}$ occuring in \req{c-n} must be
orthogonal to each other. 
In fact, under the given assumptions one can represent the double
layer integral operator as a block circulant integral operator, and
use this representation to convince oneself that for such domains $K$ has many
eigenvalues of higher multiplicities, and that the
corresponding eigenfunctions in \req{c-n} are, indeed, orthogonal.
\fin
\end{example}

\section{Polarization tensors with no more than two poles}
\label{Sec:Riemann}
We want to use Theorem~\ref{Thm:main} to completely characterize all
bounded and simply connected Lipschitz domains $\O\subset\R^2$ for
which the polarization tensor happens to be a rational function of the
entire complex plane with at most two poles, and with coefficients
$c_n$ in \req{MlO} restricted to have absolute values equal to one. 
In other words, we assume that
\be{MlE}
   M(\mu;\O) \,=
   \begin{cmatrix}
   {\displaystyle \phantom{x}
      \frac{r_+^2}{\mu-\lambda} 
             \,+\, \frac{r_-^2}{\mu+\lambda}\quad
   }  &
   {\displaystyle
      \frac{r_+r_-}{\mu-\lambda} 
             \,-\, \frac{r_+r_-}{\mu+\lambda}
      \phantom{x}
   }  \\[4ex]
   {\displaystyle \phantom{x}
      \frac{r_+r_-}{\mu-\lambda} 
             \,-\, \frac{r_+r_-}{\mu+\lambda}
      \quad
   }  &
   {\displaystyle
      \frac{r_-^2}{\mu-\lambda} 
             \,+\, \frac{r_+^2}{\mu+\lambda} \phantom{x}
   }
   \end{cmatrix},
\ee
where $r_\pm\geq 0$, (at most) one of which may be zero, and 
$\pm\lambda\in(-1/2,1/2)$. Note that $M(\mu;\O)$ either has two poles if 
$\lambda\neq 0$ or one pole when $\lambda=0$.

Choosing
\bdm
   c\,=\, \frac{r_+}{(r_+^2+r_-^2)^{1/2}} \qquad \text{and} \qquad
   s\,=\, \frac{r_-}{(r_+^2+r_-^2)^{1/2}}
\edm
we obtain
\bdm
   \begin{cmatrix}
     c & s
   \end{cmatrix}
   M(\mu;\O)
   \begin{cmatrix}
     c \\ s
   \end{cmatrix}
   \,=\, \frac{r_+^4 \,+\, 2r_+^2r_-^2 + r_-^4}{r_+^2+r_-^2}\,
         \frac{1}{\mu-\lambda}
   \,=\, \frac{r_+^2+r_-^2}{\mu-\lambda} ,
\edm
and---by virtue of \req{poltensor}---the left hand side equals
\bdm
   \int_{\partial\O} \partial_\nu u^- (\mu I-K)^{-1} u^-\ds 
   \qquad \text{for} \qquad
   u^-(x) \,=\, cx_1+sx_2\,.
\edm
Rewriting $u^-$ in $\O$ as a double layer potential with density
$\varphi$, we conclude as in \req{M11} that $S^{-1/2}\varphi$ is an
eigenfunction of $A$ for the eigenvalue $\lambda$, and hence,
$\varphi$ is an eigenfunction of the double layer operator $K$ for the
same eigenvalue. 

It follows from the jump relations that the double layer potential
\bdm
   \int_{\partial\O} \partial_{\nu(y)}\Phi(x,y)\varphi(y)\ds(y)
   \,=\, \begin{cases}
           u^-(x)=cx_1+sx_2\,, & x\in\O\,,\\
           u^+(x)\,, & x\in\R^2\setminus\overline\O\,,
         \end{cases}
\edm
satisfies
\be{utrace-Sec7}
   u^+|_{\partial\O} 
   \,=\, \bigl(K + \frac{1}{2}I\bigr) \varphi
   \,=\, \frac{2\lambda+1}{2}\,\varphi
   \,=\, \frac{2\lambda+1}{2\lambda-1}\, u^-|_{\partial\O}\,.
\ee
Thus, 
\begin{equation*}
  \utilde(x) \,=\,
  \begin{cases}
    u^-(x) \,, &x\in\O\,,\\
    u^+(x)+\dfrac{2}{1-2\lambda}(cx_1+sx_2) \,, &x\in\R^2\setminus\O\,,
  \end{cases}
\end{equation*}
solves the transmission problem
\begin{align*}
  \Delta\utilde &\,=\, 0 
  &&\text{in } \R^2\setminus\di\O\,,\\
  \utilde|^+_{\di\O} &\,=\, \utilde|^-_{\di\O} \,, \quad 
  \di_\nu\utilde|^+_{\di\O} \,=\, 
  \frac{3-2\lambda}{1-2\lambda} \,\di_\nu\utilde|^-_{\di\O} 
  &&\text{on } \di\O\,,\\
  \utilde(x) &- \frac{2}{1-2\lambda}(cx_1+sx_2) \longrightarrow 0 
  &&\text{as } |x|\to\infty \,.
\end{align*}
Since the gradient of $\utilde$ also happens to be constant in $\O$, it
follows from a variant of the strong Eshelby conjecture, which was
proved by Ru and Schiavone~\cite{RuSch96} (see also Kang and 
Milton~\cite{KaMi08}, and Liu~\cite{Liu08}), that $\O$ is an ellipse.

In the following we give a slightly extended version of the proof from
\cite{RuSch96} to show that this ellipse $\O$ is indeed completely
determined up to translations by knowing the full polarization
tensor for all admittivity contrasts; 
in fact, due to symmetry and \req{M22M11} it is sufficient to know
the first column of $M(\mu;\O)$.

Given $u^\pm$ and $\varphi$ as above, 
we consider the restrictions $v^\pm$ of the associated single layer
potential $v$ of \req{v} to
$\R^2\setminus\overline\O$ and $\O$, respectively, which provide
harmonic conjugates of $u^\pm$ with $v^+(x)\to 0$ for $|x|\to\infty$.
Accordingly,
\bdm
   v^-(x) \,=\, cx_2-sx_1 + d\,, \qquad x\in\O\,,
\edm
where $d\in\R$ is a constant, and hence, $v^+$ has boundary values
\be{vtrace-Sec7}
   v^+(x) \,=\, v^-(x) \,=\, cx_2-sx_1 + d\,, \qquad x\in\partial\O\,.
\ee

By the Riemann mapping theorem there is a unique
conformal transformation $\PsiG$ that takes the exterior of the unit disk 
$\D\subset\Chat$ onto $\Chat\setminus\overline\O$, and satisfies
\bdm
   \PsiG(\infty)\,=\,\infty\,, \qquad \gamma\,=\, \PsiG'(\infty)\,>\,0\,.
\edm
The parameter $\gamma$ is called the capacity of $\O$.
$\PsiG$ has a Laurent expansion of the form
\be{Psi}
   \PsiG(\zeta) \,=\, \gamma \zeta + 
                  \sum_{k=1}^\infty a_k\zeta^{-k}\,, \qquad |\zeta|\geq 1\,,
\ee
with coefficients $a_k\in\C$, $k\in\N$;
note that we can assume without loss of generality that no zero order term
occurs, since the exact position of $\O$ does not enter into the
polarization tensor.

Consider now the analytic function
\bdm
   F(z) \,=\, u^+(z)+\rmi\;\!v^+(z)\,, \qquad z\in\Chat\setminus\O\,,
\edm
which is bounded and satisfies $F(z)\to 0$ for $z\to \infty$. 
By virtue of \req{utrace-Sec7} and \req{vtrace-Sec7} its boundary
values are 
\bdm
   F(z) \,=\, -q\,\Real(e^{-\rmi\phi}z) \,+\, \rmi\,\Imag(e^{-\rmi\phi}z)
              \,+\, \rmi\;\!d
   \,=\, \frac{1-q}{2}\,e^{-\rmi\phi}z 
         \,-\, \frac{1+q}{2}\,e^{\rmi\phi}\bar{z} \,+\, \rmi\;\!d
\edm
on $\partial\O$, where
\bdm
   q\,=\, \frac{1+2\lambda}{1-2\lambda} \,>\,0 \qquad \text{and} \qquad
   e^{\rmi\phi}\,=\, c+\rmi\;\!s\,.
\edm
It follows that $G=F\circ\PsiG$ is a bounded analytic function in
$\Chat\setminus\D$ with a zero at infinity and limiting values
\bdm
   G(e^{\rmi\theta}) 
   \,=\, \frac{1-q}{2}\,e^{-\rmi\phi}\PsiG(e^{\rmi\theta})
         \,-\, \frac{1+q}{2}\, e^{\rmi\phi}\overline{\PsiG(e^{\rmi\theta})}
         \,+\, \rmi\;\!d\,,
   \qquad 0\leq\theta<2\pi\,,
\edm
on $\partial\D$.
Inserting the Laurent series~\req{Psi} of $\PsiG$ we obtain
\be{Ghelp}
\begin{aligned}
   G(e^{\rmi\theta})
   &\,=\, -\sum_{k=2}^\infty 
              \frac{1+q}{2}e^{\rmi\phi}\,\overline{a_k} \,e^{\rmi k\theta}
          \,+\, \Bigl(\frac{1-q}{2} e^{-\rmi\phi} \gamma
                      \,-\, \frac{1+q}{2}e^{\rmi\phi}\,\overline{a_1} 
                \Bigr) e^{\rmi\theta} 
          \,+\, \rmi\;\!d\\[1ex]
   &\qquad \quad
          \,+\, \Bigl(\frac{1-q}{2} e^{-\rmi\phi} a_1
                      \,-\, \frac{1+q}{2}e^{\rmi\phi} \gamma
                \Bigr) e^{-\rmi\theta}
          \,+\, \sum_{k=2}^\infty 
                   \frac{1-q}{2} e^{-\rmi\phi} a_k \,e^{-\rmi k\theta}\,.
\end{aligned}
\ee
On the other hand, since $G:\Chat\setminus\D\to\C$ is a bounded analytic
function with $G(\infty)=0$, its Laurent series has the form
\bdm
   G(\zeta) \,=\, \sum_{k=1}^\infty b_k \zeta^{-k}\,,
\edm
and inserting $\zeta=e^{\rmi\theta}$ and comparing the result with
\req{Ghelp} it follows that
\bdm
   \frac{1+q}{2}\,e^{\rmi\phi}\,\overline{a_k} \,=\, 0 \qquad
   \text{for all $k\geq 2$}\,,
\edm
and
\bdm
   \frac{1-q}{2} e^{-\rmi\phi} \gamma
   \,-\, \frac{1+q}{2}e^{\rmi\phi}\,\overline{a_1} \,=\, 0\,.
\edm
Since $q>0$ this implies that $a_k=0$ for all $k\geq 2$, and that
\bdm
   a_1 \,=\, \frac{1-q}{1+q} e^{\rmi2\phi}\gamma
   \,=\, -2\lambda e^{\rmi2\phi}\gamma\,,
\edm
i.e.,
\bdm
   \PsiG(\zeta) 
   \,=\, \gamma e^{\rmi\phi}
         \bigl( (e^{-\rmi\phi}\zeta)\,-\,2\lambda(e^{-\rmi\phi}\zeta)^{-1}
         \bigr)\,.
\edm 
This shows that $\PsiG$ is a Joukowski transformation which takes the
unit circle onto an ellipse with eccentricity $\max\{q,1/q\}$ 
(resp.\ a disk, when $\lambda=0$) centered at the origin, whose 
axes have polar angles $\phi$ and $\phi+\pi/2$.
Hence, $\O$ is the interior of this ellipse,
the volume of which is $|\O|=r_+^2+r_-^2$ by virtue of \req{Mkklimit}.

We summarize our findings in the following theorem.

\begin{theorem}
\label{Thm:ellipse}
If $M(\mu;\O)$ is given by \req{MlE} then $\O$ is an ellipse whose half axes
of lengths
\bdm
   a\,=\, \frac{(r_+^2+r_-^2)^{1/2}}{\sqrt\pi} 
          \Bigl(\frac{1-2\lambda}{1+2\lambda}\Bigr)^{1/2}
   \qquad \text{and} \qquad
   b\,=\, \frac{(r_+^2+r_-^2)^{1/2}}{\sqrt\pi} 
          \Bigl(\frac{1+2\lambda}{1-2\lambda}\Bigr)^{1/2}
\edm
make angles $\phi=\arctan(r_-/r_+)$ and $\phi+\pi/2$ to the horizontal axis,
respectively.
\end{theorem}

\begin{remark}
\label{Rem:equivalentellipses}
\rm
In the literature (e.g., in \cite[Sect.~4.11.1]{AmKa07}) the polarization
tensor of an unknown domain $\O$ for a fixed admittivity contrast $\mu$ 
is sometimes used to determine a so-called \emph{equivalent ellipse} $\E$ 
that shares this particular polarization tensor. 
It follows from the above derivation that this ellipse, in fact,
changes with $\mu$---unless $\O$ \emph{is} an ellipse, in which case
$\O=\E$. 

For, if the ``equivalent ellipse'' would not change with $\mu$,
(say, taken from a countable set $\{\mu_k\}$ that clusters somewhere
in $\Chat\setminus(-1/2,1/2)$)
then the polarization tensors of $\E$ and $\O$ would coincide 
for these arguments and by the uniqueness theorem for analytic functions 
the polarization tensor of $\O$ would be of the same form~\req{MlE}
as the one of $\E$. 
As we have proved, this implies $\O=\E$.
\fin
\end{remark}

\begin{remark}
\label{Rem:ACLZ13}
\rm
Given the result of this section one might ask oneself whether the knowledge
of the polarization tensor $M(\mu;\O)$ as a function of the 
admittivity contrast suffices to determine the shape of a general 
simply connected and bounded Lipschitz domain $\O$.
In fact, the best possible result to expect is that this information
determines $\O$ up to reflections, because $M(\mu;-\O)=M(\mu;\O)$ by virtue
of \req{similarity}.

In this context the authors of \cite{ACLZ13} numerically optimize 
the shape of a target domain to fit 
the eigenvalues of the corresponding double layer operator $K$ to the poles 
of a given polarization tensor. Since these poles only constitute a subset
of the eigenvalues of $K$, in general (see the ellipse as an example), 
this can be a very underdetermined problem to solve.
Moreover, Schiffer~\cite{Schi57} has proved that the eigenvalues of $K$ are
invariant under M\"obius transformations of $\partial\O$; therefore 
knowledge of the poles alone does not suffice to completely determine $\O$.
\fin
\end{remark}

\section{The isoperimetric inequalities}
\label{Sec:inequalities}
In this section we revisit (in two space dimensions) the so-called 
isoperimetric inequalities, also known as Hashin-Shtrikman bounds,
\begin{subequations}
\label{eq:iso}
\begin{align}
\label{eq:iso1}
   \frac{2}{|\mu|}\,|\O|
   \,\leq\,  \bigl|\,\trace\bigl(M(\mu;\O)\bigr)\,\bigr| 
   &\,<\, \frac{8|\mu|}{4\mu^2-1}\,|\O|\,,  \\[1ex]
\label{eq:iso2}
   \bigl|\,\trace\bigl(M(\mu;\O)^{-1}\bigr)\,\bigr|
   &\,\leq\, \frac{2|\mu|}{|\O|}\,, 
\end{align}
\end{subequations}
valid for $\mu\in\R\setminus(-1/2,1/2)$.
We refer to \cite{AmKa07} and \cite{Milt02}
for the background of these inequalities,
and for variational proofs of them (see also
\cite{BeKo93,CapVog06,KoMi86,Lip93}).
In the sequel we will provide 
a proof of these inequalities on the grounds of Theorem~\ref{Thm:main},
and discuss the equality signs in \req{iso}.
As in Theorem~\ref{Thm:main}, 
the only assumption on $\O$ is to be a simply connected Lipschitz domain.

We start from the representations~\req{M11} and \req{M22M11} of the
diagonal elements of the polarization tensor, and first note that
\be{Flaeche}
   \int_{-\infty}^\infty \da_\lambda \,=\, |\O|
\ee
by virtue of \req{Mkklimit}. From \req{M11} and \req{M22M11} we further
conclude that the trace of $M(\mu;\O)$ is given by
\bdm
   \trace\bigl(M(\mu;\O)\bigr) 
   \,=\, \int_{-\infty}^\infty \frac{2\mu}{\mu^2-\lambda^2} \da_\lambda \,.
\edm
Since $\alpha\ll\asymm$ (cf.~\req{estgamma}), 
and since the support of $\asymm$ is contained in $(-1/2,1/2)$ by virtue of
Theorem~\ref{Thm:main}, it follows that
\bdm
   \frac{2}{|\mu|}\int_{-\infty}^\infty \da_\lambda
   \,\leq\, \bigl|\,\trace\bigl(M(\mu;\O)\bigr)\,\bigr| 
   \,<\, \frac{2|\mu|}{\mu^2-1/4}\int_{-\infty}^\infty \da_\lambda
\edm
for every $\mu\geq1/2$, and the same argument applies to values of $\mu$
less or equal than $-1/2$. Making use of \req{Flaeche} we thus obtain
the first isoperimetric inequality~\req{iso1}.

The above derivation reveals that the lower bound in \req{iso1}
is attained for any value of
$\mu$ with $|\mu|\geq 1/2$, if and only if $\alpha$
is supported in the origin, that is, if $M(\mu;\O)$ is 
a rational function with a simple pole at $\lambda=0$.
As we have discussed in Section~\ref{Sec:Riemann}, 
this is the case, if and only if $\O$ is a disk;
this is known as a conjecture of P\'olya and Szeg\H{o}.
In contrast, the upper bound is never attained:
that the upper bound cannot be attained by domains $\O$ with a certain
``thickness'' has been shown by Capdeboscq and Kang~\cite{CaKa06}.

Concerning the trace of $M(\mu_*;\O)^{-1}$ 
for a given $\mu_*\in\R\setminus(-1/2,1/2)$, we first rotate the coordinate
system in such a way that the orthogonal eigenbasis of 
(the real symmetric matrix) $M(\mu_*;\O)$ is parallel to the two coordinate 
axes. This means that the off-diagonal entries of the polarization tensor
(in the rotated coordinate system) have a root at $\mu=\mu_*$, and hence, 
\bdm
   \trace\bigl(M(\mu_*;\O)^{-1}\bigr)
   \,=\, \frac{1}{f_+(\mu_*)} \,+\, \frac{1}{f_-(\mu_*)}\,,
\edm
where
\bdm
   f_\pm(\mu) \,=\, \int_{-\infty}^\infty \frac{1}{\mu\pm\lambda} \da_\lambda
\edm
are the diagonal entries \req{M11} and \req{M22M11} of the 
polarization tensor in the rotated coordinates.
Straightforward differentiation reveals that $(1/f_\pm)'(\mu)>0$ for 
$\mu \geq 1/2$ and
\bdmal
   \Bigl(\frac{1}{f_\pm}\Bigr)''(\mu) 
   &\,=\, 2\,\Bigl(\int_{-\infty}^\infty 
                         \frac{1}{\mu\pm\lambda}\,\da_\lambda\Bigr)^{-3}
          \cdot\\[1ex]
   &\qquad 
          \left\{\Bigl( \int_{-\infty}^\infty 
                            \frac{1}{(\mu\pm\lambda)^2}\,\da_\lambda
                 \Bigr)^2
            \,-\,\int_{-\infty}^\infty \frac{1}{(\mu\pm\lambda)^3}\,\da_\lambda
                 \int_{-\infty}^\infty \frac{1}{\mu\pm\lambda}\,\da_\lambda
          \right\},
\edmal
where the term in curled braces on the right is nonpositive by virtue
of the Cauchy-Schwarz inequality; to be more precise, the second
derivative is strictly negative, unless the measure $\alpha$ is
concentrated on one single point in the interval $(-1/2,1/2)$.

It follows that 
\bdm
   g(\mu) \,=\, \frac{1}{f_+(\mu)} \,+\, \frac{1}{f_-(\mu)}
\edm
is concave for $\mu\geq 1/2$. Moreover, in view of \req{Flaeche} there holds
\bdm
   \frac{1}{f_{\pm}(\mu)} 
   \,=\, \frac{\mu}{|\O|}
         \,\pm\, \frac{1}{|\O|^2} \int_{-\infty}^\infty \lambda\da_\lambda
         \,+\, O(\mu^{-1}) \,, \qquad \mu\to\infty\,,
\edm
and hence, 
\bdm
   g(\mu) \,=\, \frac{2\mu}{|\O|} \,+\, O(\mu^{-1})\,, \qquad \mu\to\infty\,.
\edm
Since $g$ is concave this implies that
\be{gbound}
   g(\mu) \,\leq\, \frac{2\mu}{|\O|} \qquad 
   \text{for all $\mu\geq 1/2$}\,,
\ee
with equality for any single value of $\mu$, if and only if $\alpha$ is
concentrated on a single point $\lambda\in(-1/2,1/2)$.
If this happens to be the case then the polarization tensor has the 
representation
\bdm
   M(\mu;\O) \,=
   \begin{cmatrix}
      {\displaystyle \phantom{x} \frac{r^2}{\mu-\lambda} \quad } 
      &
      0 \phantom{x} \\[2ex]
      \phantom{x} 0 \quad
      &
      {\displaystyle \frac{r^2}{\mu+\lambda} \phantom{x} }
   \end{cmatrix}
   \qquad \text{for all $\mu\in\C$}
\edm
in the rotated coordinate system,
and hence, $\Omega$ is an ellipse (or a circle, when $\lambda=0$), as we
have seen in Section~\ref{Sec:Riemann}.
Note that \req{gbound} immediately implies that
\bdm
   |g(\mu)| \,\leq\, \frac{2|\mu|}{|\O|} \qquad
   \text{for all $\mu\in\R\setminus(-1/2,1/2)$}\,,
\edm
because $f_+(-\mu)=-f_-(\mu)$ for every $\mu\in\C$ (cf.~\req{M22M11}).

We thus have established the other isoperimetric
inequality~\req{iso2}, and have shown that equality holds for any
value of $\mu\in\R\setminus(-1/2,1/2)$,  if and only if $\O$ is an
ellipse; the latter was proved first by Kang and
Milton~\cite{KaMi06,KaMi08}.

\section{Concluding remarks}
\label{Sec:Conclusion}
We have derived an analytic representation of polarization
tensors as function of the admittivity contrast in terms of the
spectral decomposition of the double layer integral operator
associated with the underlying domain. 
Since our arguments rely on complex analysis and layer potential
techniques, a generalization of this result to three-dimensional objects
or more complicated non-constant conductivity inhomogeneities
(see, e.g., Capdeboscq and Vogelius~\cite{CapVog03a}) is not
straightforward. 

We have considered two applications of this analytic representation,
namely (i) an elementary proof of the Hashin-Shtrikman
bounds for the trace of polarization tensors and their inverses in
terms of the area of the associated domain, and (ii) we have
established a one-to-one correspondence between ellipses and certain
polarization tensors with at most two poles as a function of the
admittivity contrast.
It remains an open problem to what extent the shape of a general planar
simply connected and bounded Lipschitz domain is determined by the
polarization tensor as a function of the admittivity contrast.

Finally we note that the results of this work have further been utilized in
\cite{GrHa14b} to analyze a multi-frequency MUSIC-type method for
electrical impedance tomography.

\section*{Acknowledgements}
We like to thank Grame W. Milton (University of Utah) for pointing out to 
us relevant literature from the theory of composite materials.


\end{document}